\documentclass{amsart}

\usepackage{amsmath,amssymb,amscd} \usepackage{graphicx}
\DeclareGraphicsExtensions{.eps} \usepackage{mathrsfs}
\usepackage[mathcal]{eucal}

\newtheorem{Thm}{Theorem}{\bfseries}{\itshape}
\newtheorem{Cor}{Corollary}{\bfseries}{\itshape}
\newtheorem{Prop}[Cor]{Proposition}{\bfseries}{\itshape}
\newtheorem{Lem}[Cor]{Lemma}{\bfseries}{\itshape}
\newtheorem{Conj}[Cor]{Conjecture}{\bfseries}{\itshape}
\newtheorem{Def}[Cor]{Definition}{\bfseries}{\rmfamily}
{\scshape}{\rmfamily}
\newtheorem{Rem}[Cor]{Remark}{\scshape}{\rmfamily}

\newtheorem{Transformation}{Transformation}{\bfseries}{\itshape}

\renewcommand\ge{\geqslant} \renewcommand\le{\leqslant}
\let\tildeaccent=\~ \let\hataccent=\^
\renewcommand\~[1]{\widetilde{#1}} \renewcommand\^[1]{\widehat{#1}}

\def\<{\left<} \def\>{\right>} \def\({\left(} \def\){\right)}

\def\abs#1{\left\vert #1 \right\vert} \def\size#1{\mathbf S\left(#1
  \right)}

\let\polishL=l \def\Zoladek.{\.Zol\c adek}

\def\Var{\operatorname{Var}} 
 \def\Mat{\operatorname{Mat}}
 
\def\Re{\operatorname{Re}} \def\Im{\operatorname{Im}}
\def\Arg{\operatorname{Arg}} 
 \def\ord{\operatorname{ord}}
\def\GL{\operatorname{GL}} \def\etc.{\emph{etc}.}
 
\def\:{\colon} \def\R{{\mathbb R}} \def\C{{\mathbb C}} \def\Z{{\mathbb
    Z}} \def\N{{\mathbb N}} \def\Q{{\mathbb Q}} 
 \def\e{\varepsilon} 
\def\l{\lambda}   
 \def\diag{\operatorname{diag}}
\def\poly{\operatorname{\textup{Poly}}}

 \def\d{\,\mathrm d}

 \let\PolishL=\L \def\Lojas.{\PolishL ojasiewicz}
\def\L{\varLambda} \def\cN{{\mathcal N}} \def\cH{{\mathcal H}}
\def\cP{{\mathcal P}}

\begin{document}

\title{A uniform version of the Petrov-Khovanskii theorem} \author{Gal
  Binyamini \and Gal Dor}


\begin{abstract}
  An Abelian integral is the integral over the level curves of a
  Hamiltonian $H$ of an algebraic form $\omega$. The infinitesimal
  Hilbert sixteenth problem calls for the study of the number of zeros
  of Abelian integrals in terms of the degrees $H$ and
  $\omega$. Petrov and Khovanskii have shown that this number grows at
  most linearly with the degree of $\omega$, but gave a purely
  existential bound. Binyamini, Novikov and Yakovenko have given an
  \emph{explicit} bound growing doubly-exponentially with the degree.

  We combine the techniques used in the proofs of these two results,
  to obtain an explicit bound on the number of zeros of Abelian
  integrals growing linearly with $\deg\omega$.
\end{abstract}

\thanks{The work was supported by the ISF grant 493/09 and the
  Legacy-Heritage foundation mentoring program}

\maketitle

\section{Introduction}
\label{sec:introduction}

Let $H$ be a real bivariate polynomial and $\omega$ a one-form on
$\R^2$. Let $\delta_t\subseteq\{H=t\}$ denote a continuous family of
real ovals. Consider the \emph{Abelian integral}
\begin{equation}
  I_{H,\omega}(t)=\int_{\delta_t}\omega
\end{equation}
The \emph{infinitesimal} Hilbert Sixteenth problem calls for the study
of the zero set
\begin{equation}
  Z_{H,\omega}=\{t:I_{H,\omega}(t)=0\}.
\end{equation}
In particular, the goal is to obtain an upper bound
$\cN(\deg H,\deg\omega)$ on $\# Z_{H,\omega}$ depending solely on the
degrees of $H$ and $\omega$. Here and in the rest of the paper $\# A$
denotes the number of isolated points in the set $A$.

The infinitesimal Hilbert problem is motivated by the study of
\emph{limit cycles} born from the perturbation $\d H+\e\omega=0$ of
the Hamiltonian system $\d H=0$. In particular, the existence of the
uniform bound $\cN(\deg H,\deg\omega)$ may be seen as a particular
case of the general Hilbert sixteenth problem. We refer the reader to
the surveys \cite{centennial,h16survey} for further
details and references.

\subsection{Background}

The first \emph{general} result concerning the infinitesimal Hilbert
problem was given in \cite{varchenko:finiteness,asik:finiteness}:
\begin{Thm}\label{thm:vk}
  \begin{equation}
    \cN(n,m) < \infty
  \end{equation}
\end{Thm}
In other words, the number of zeros of Abelian integral is uniformly
bounded in terms of the degrees of the Hamiltonian $H$ and the form
$\omega$. However, this result is purely existential and does not give
an explicit bound for $\cN(n,m)$.

The following uniform upper bound, established in \cite{inf16}, constitutes
an explicit solution for the infinitesimal Hilbert problem.

\begin{Thm}\label{thm:bny}
  \begin{equation}\label{eq:bny-bound}
    \cN(n,n)\le2^{2^{\poly(n)}}
  \end{equation}
  where $\poly(n)$ denotes an explicit polynomial of degree not
  exceeding 61.
\end{Thm}

We call the attention of the reader to the fact that the dependence of
the upper bound \eqref{eq:bny-bound} is \emph{doubly-exponential} in
both $\deg H$ and $\deg\omega$. In contrast, Petrov and Khovanskii
proved the following result in an unpublished work (see \cite{Zol,roitman-msc}
for an exposition).

\begin{Thm}\label{thm:pk}
  \begin{equation}
    \# Z_{H,\omega} \le a(n)m+b(H)
  \end{equation}
  where $n=\deg H,m=\deg\omega$, with $a(n)$ some explicit function
  and $b(H)$ \emph{some} function of $H$ (for which a bound is not
  given).
\end{Thm}

The bound given by Theorem~\ref{thm:pk} is not uniform over the class
of Hamiltonians of a given degree, due to the appearance of the term
$B(H)$. However, using the methods developed in the proof of
Theorem~\ref{thm:vk} it is possible to prove that this term is in fact
uniformly bounded \cite{Zol}.

\begin{Thm}\label{thm:pk-z}
  \begin{equation}
    \cN(n,m) \le a(n)m+b(n)
  \end{equation}
  where $n=\deg H,m=\deg\omega$, with $a(n)$ some explicit funcition
  and $b(n)$ \emph{some} function of $n$ (for which a bound is not
  given).
\end{Thm}

In summary, Theorem~\ref{thm:bny} establishes an \emph{explicit} bound
on $\cN(n,m)$ depending \emph{doubly-exponentially} on $m$, whereas
Theorem~\ref{thm:pk-z} establishes an \emph{existential} bound
depending \emph{linearly} on $m$. The goal of this paper is to apply a
combination of the ideas used in the proofs of these two results, to
obtain an \emph{explicit} bound depending \emph{linearly} on $m$.

The result is as follows. We introduce the following notation to
simplify the presentation of the results. We write $O^+(f(x))$ as a
shorhand for $O(f(x)\log f(x))$. We write $\exp(x)$ for
$2^x$, and $\exp^+(x)$ for $\exp(O^+(x))$. Finally we allow
compositional iteration in the usual way, so $\exp^2(x)$ corresponds
to $2^{2^x}$, etc.

\begin{Thm}\label{thm:main-ai}
  \begin{equation}
    \cN(n,m) \le \exp^{+2}(n^2)\cdot m+\exp^{+5}(n^2)
  \end{equation}
\end{Thm}

See subsection~\ref{sec:concluding-remarks} for a discussion of the
percise form of the bound and possible improvements.

\section{Preliminaries and setup}

In this section we review background from the theory of analytic
differential equations, the theory of Abelian integrals, and the work
\cite{inf16}.

\subsection{Connections, integrability and regularity}
\label{sec:connections}

Let $\Omega\in\Mat(\ell,\Lambda^1(\C P^m))$ denote an $\ell\times\ell$
rational matrix one-form over $\C P^m$ with singular locus
$\Sigma$. The form is said to be \emph{integrable} if
$\d\Omega=\Omega\wedge\Omega$. This condition is equivalent to the
existence of a fundamental solution matrix $X(\cdot)$, defined over
$\C P^m\setminus\Sigma$ and ramified over $\Sigma$, for the following
system of equations
\begin{equation}\label{eq:connection-equation}
  \d X = \Omega \cdot X.
\end{equation}
In other words, we view $\Omega$ as the matrix form of a connection
defined on the trivial $\ell$-dimensional vector bundle over $\C P^m$
and $X$ as a fundamental matrix of horizontal sections.

Let $(\l_1,\ldots,\l_m)$ denote an affine chart on $\C P^m$, and for
convenience of notation let $t=\l_1,\l'=(\l_2,\ldots,\l_m)$. Then the
system~\eqref{eq:connection-equation} may be viewed as a family of
linear systems of differential equations in the $t$ variable
parameterized by $\l'$,
\begin{equation}\label{eq:family-equation}
  \frac{\d X}{\d t} = \Omega_{\l'}(t) X(t).
\end{equation}
We remark that not every system of the form \eqref{eq:family-equation}
may be obtained in this manner. In particular, systems obtained in
this manner are necessarily \emph{isomonodromic}.

The system~\eqref{eq:connection-equation} is said to be \emph{regular}
if for any germ of a \emph{real analytic} path
$\gamma:(\R,0)\to\C P^m$ with
$\gamma(\R\setminus\{0\})\subseteq\C P^m\setminus\Sigma$, the rate of
growth of the fundamental solution matrix along $\gamma$ is
polynomial. Explicitly, we require that for suitable positive
constants $c,k$ we have
\begin{equation}
  \abs{X(\gamma(s))}^{\pm 1} \le c\abs{s}^{-k} \qquad \forall s\in(\R,0).
\end{equation}
The analyticity of the curve $\gamma$ is required to rule out
spiralling around the singular locus.

\subsection{Monodromy and Quasi-Unipotence}
\label{sec:monodromy}

To each closed loop $\gamma\in\C P^m\setminus\Sigma$ one may associate
a continuation operator $\Delta_\gamma$ describing the result of
analytic continuation of $X(\cdot)$ along $\gamma$. The
\emph{monodromy matrix} $M_\gamma\in\GL(\ell,\C)$ is defined by the
equation $\Delta_\gamma X=X\cdot M_\gamma$. It is clear that
$M_\gamma$ depends only on the pointed homotopy class of $\gamma$, and
that the conjugacy class of $M_\gamma$ depends only on the free
homotopy class of $\gamma$. In the future we shall mainly be
interested in the conjugacy class of the monodromy, and refer to the
monodromy associated with a homotopy class of a closed loop in this
sense.

A matrix $M$ is said to be \emph{quasi-unipotent} if all of its
eigenvalues are roots of unity. Equivalently, $M$ is quasi-unipotent
if and only if there exist $j,k\in\N$ such that $(M^j-I)^k=0$, where
$I$ denotes the identity matrix. We shall say that the monodromy along
a loop $\gamma$ is quasi-unipotent if the associated monodromy matrix
$M_\gamma$ is quasi-unipotent (note that this condition depends only
on the conjugacy class of $M_\gamma$).

A loop $\gamma$ is said to be a \emph{small loop} around $\l_0$ if
there exists a germ of an analytic curve $\tau:(\C,0)\to(\C P^m,\l_0)$
with $\tau(\C\setminus\{0\})\subseteq\C\setminus\Sigma$ such that
$\gamma$ is homotopic to a closed path $\tau(\{\abs{z}=\e)$ for
sufficiently small $\e$. We shall only be interested in the case
$\l_0\in\Sigma$.

The system~\eqref{eq:connection-equation} is said to be
\emph{quasiunipotent} if the monodromy matrix associated to each
\emph{small loop} is quasi-unipotent. Note that this condition does
\emph{not} imply that \emph{every} monodromy matrix associated with
the system is quasi-unipotent. In particular, monodromies along loops
encircling several singualities are often not small, and are not
required to be quasi-unipotent (and this is indeed the case in natural
examples).

\subsection{Complexity of algebraic objects}
\label{sec:complexity}

In this subsection we give definitions for measuring the
complexity of the formulas representing various algebraic objects. It
is rather unusual in mathematics to be concerned with the particular
formulas used for the description of an object. Questions of this form
fall more neatly within the framework of \emph{mathematical
  logic}. Indeed, strictly speaking the definitions in this subsection
could be more accurately expressed in terms of logical formula
complexity. In the interest of simplicity we content ourselves with
simple algebraic approximations of these notions which are sufficient
for our purposes.

We stress that all definitions in this subsection refer to a
particular representation of a given object. For instance, $x^2/x$ and
$x/1$ are viewed as distinct fractional representations of the same
polynomial.

A polynomial $P\in\Z[x_1,\ldots,x_n]$ is said to be a \emph{lattice}
polynomial. We shall say that such a polynomial is \emph{defined over
  $\Q$}, if
\begin{equation}
  P(x_1,\ldots,x_n) = \sum_\alpha c_\alpha x^\alpha \qquad c_\alpha\in\Z
\end{equation}
where $\alpha$ denotes a multiindex. We define the size of $P$ to be
$\size{P}=\sum_\alpha \abs{c_\alpha}$.

A rational function given by a fraction of the form $P/Q$ is said to
be defined over $\Q$ if $P$ and $Q$ are defined over $\Q$. In this
case, we define the size $\size{P/Q}$ to be $\size{P}+\size{Q}$.

Similarly, a one-form $\omega$ is said to be defined over $\Q$ if it
is of the form
\begin{equation}
  \omega = \sum_i R_i(x) \d x_i
\end{equation}
where $R_i$ are rational functions defined over $\Q$. In this case, we
define the size $\size{\omega}$ to be $\sum_i \size{R_i}$.

Finally, say that a vector or a matrix is defined over $\Q$ if its of
its components are, and define the size to be the sum of the sizes of
components.

\subsection{Counting zeros of multivalued vector functions}
\label{sec:multivalued-zeros}

Recall that we may view the system~\eqref{eq:connection-equation} as a
family of differential equations in the variable $t$, of the
form~\eqref{eq:family-equation}. We shall be interested in studying
the oscillatory behavior of the solutions of this equation. However,
due to the fact that the solutions of \eqref{eq:family-equation} may
be ramified, some care is required in measuring this oscillation.

Let $f$ be a (possibly multivalued) function defined in a domain
$U\subseteq\C$. If $U$ is simply connected, then we define the
following \emph{counting function} as a measure for the number of
zeros of $f$:
\begin{equation}
  \cN_U(f) = \sup_{b} \#\{t:b(t)=0\},
\end{equation}
where $b$ varies over the branches of $f$ in $U$ (which are well
defined univalued functions, since $U$\ is simply connected).

For general domains, we use the following counting function,
\begin{equation}
  \cN_U(f) = \sup_{T\subseteq U} \cN_T(f),
\end{equation}
where $T$ varies over all triangular domains (i.e., domains whose
boundary consists of straight line segment). The restriction on the
geometry of $T$ is needed in order to avoid spiralling around a
singular point. We stress that the closure of $T$ \emph{need not} be
contained in $U$. The boundary may contain singular points. When
$U$ is omitted from the notation, it is understood to be the
domain of analyticity of the function $f$.

Let $L$ be a linear space of (possibly multivalued) functions defined
in a domain $U\subseteq\C$. As a measure for the number of zeros of an
element of $L$, we use the following,
\begin{equation}
  \cN_U(L) = \sup_{f\in L} \cN_U(f).
\end{equation}
When $U$ is omitted from the notation, it is understood to be the
common domain of analyticity of the elements of $L$.

\begin{Rem}[Semicontinuity]\label{rem:semicontinuity}
  As remarked in \cite{inf16}, the counting function $\cN(\cdot)$ is lower
  semicontinuous with respect to the space $L$. In particular, if we
  have a family of spaces $L_\nu$ continuously depending on a paramter
  $\nu$, then an upper bound $\cN(L_\nu)<M$ for $\nu$ in a dense
  subset of the paramter space implies the same upper bound for every
  $\nu$.
\end{Rem}

We now consider the oscillations of vector-valued solutions of the
system~\eqref{eq:family-equation}. Fix $\l'$ such that the affine line
$A=\C P^1\times\{\l'\}$ is not contained in $\Sigma$. Then $A$
intersects $\Sigma$ in finitely many points. Let $U$ denote the
complement of this intsection.

Since the system~\eqref{eq:family-equation} is non-singular in $U$, it
admits an $\ell$-dimensional space $L_{\l'}^V(\Omega)$ of (possibly
multivalued) vector-valued solution functions. To measure the
oscillation of these solutions, we shall consider the number of
intersections of a solution with an arbitrary fixed linear
hyperplane. Formally, we define the linear space
\begin{equation}\label{eq:solution-space}
  L_{\l'}(\Omega) = \{ c \cdot f : c\in\C^\ell, f\in L_{\l'}^V \}
\end{equation}
and the corresponding counting function
\begin{equation}
  \cN(\Omega) = \sup_{\l'} \cN(L_{\l'}).
\end{equation}
When the system $\Omega$ is clear from the context, we sometimes omit
it from the notation and write $L_\l'$.

We note that the counting function may in general be infinite. We also
remark that by triangulation, one may use to counting function
$\cN(\cdot)$ to study the oscillation in more complicated domains.

\subsection{Q-systems and Q-functions}
\label{sec:q-systems}

In this subsection we introduce a class of systems of the
form~\eqref{eq:connection-equation} for which explicit bounds on the
counting function $\cN(\Omega)$ may be derived. This class constitutes
the main object of study of the paper \cite{inf16}.

\begin{Def}[Q-System]
  The system~\eqref{eq:connection-equation} is said to be an
  $(s,m,d,\ell$)-\emph{Q-system} if $\Omega$ is an $\ell\times\ell$
  matrix one-form defined over $\C P^m$ such that the following holds:
  \begin{enumerate}
  \item $\Omega$ is integrable.
  \item $\Omega$ is regular.
  \item $\Omega$ is quasi-unipotent.
  \item $\Omega$ is defined over $\Q$, has size $s$, and coefficients
    of degree bounded by $d$.
  \end{enumerate}
  Functions from the corresponding linear spaces $L_{\l'}(\Omega)$ are
  said to be \emph{Q-functions}.
\end{Def}

The main interest in this class of systems stems from the following
result of \cite[Theoren 8]{inf16}, which plays the central role in the proof of
Theorem~\ref{thm:bny}.

\begin{Thm}\label{thm:bny-qs}
  Let $\Omega$ be an $(s,m,d,\ell)$-Q-system. Then we have the
  following explicit bound,
  \begin{equation}
    \cN(\Omega) \le s^{2^{\poly(m,d,\ell)}}
  \end{equation}
  where $\poly(m,d,\ell)\le O^+(d\ell^4 m)^5$.
\end{Thm}

We will also require a result concerning the order of a Q-function
near a singular point. Fix $\l'$ and let $f(t)\in L_{\l'}$ and
$(t_0,\l')\in\Sigma$ a singular point of $\Omega$. Then, since
$\Omega$ is regular and quasiunipotent, $f(t)$ admit an expansion
\begin{equation}
  f(t) = p(\ln(t-t_0)) t^\mu + o(t^\mu) \qquad p\in\C[v], \mu\in\R.
\end{equation}
We call $\mu$ the \emph{order} of $f$ at $t=t_0$, and denote
$\mu=\ord_{t_0} f$. If $\gamma_\e$ denotes a circular arc of
radius $\e$ and angle $\alpha$ around $t_0$, then
\begin{equation}\label{eq:order-lim}
  \lim_{\e\to0} \Var\Arg f(t)\big|_{\gamma_\e} = 2\alpha\mu.
\end{equation}

The following proposition follows in a straightforward manner from the proof of
Theorem~\ref{thm:bny-qs}.

\begin{Prop}\label{prop:bny-order}
  Let $\Omega$ be an $(s,m,d,\ell)$-Q-system. Fix some $(t_0,\l')\in\Sigma$ and let
  $f\in L_{\l'}(\Omega)$. Then we have the following explicit bound,
  \begin{equation}
    \abs{\ord_{t_0} f} \le s^{(d\ell)^{O(m)}}.
  \end{equation}
\end{Prop}
\begin{proof}
  By~\eqref{eq:order-lim} it suffices to estimate the variation of
  argument of $f(t)$ along $\gamma_\e$ (in absolute value). We list the appropriate
  references to \cite{inf16}. The
  estimate follows immediately from Principal Lemma 33 and Lemma 42,
  noting the the normalized length of $\gamma_\e$ approaches $2\pi$ as
  $\e\to0$. We remark that the bound of Lemma 42 is stated for the
  variation of argument of $f$, but it in fact applies to the absolute
  value of the variation of argument as well (as is easily seen from
  the proof).
\end{proof}

\subsection{Abelian integrals and the Gauss--Manin connection}
\label{sec:Abelian-integrals}

In order to apply the theory of Q-systems, and in particular
Theorem~\ref{thm:bny-qs} to the study of Abelian integrals, it is
necessary to produce a Q-system that they satisfy. The existence of
such systems goes back to Picard--Fuchs (in the
form~\eqref{eq:family-equation}), and to Gauss--Manin (in the
form~\eqref{eq:connection-equation}). Explicit derivations of this
system (in the sense of subsection~\ref{sec:complexity}) were given in
\cite{redundant,bdd-decomposition-pf}. For the convenience of the reader,
we reproduce the relevant parts of the construction below. For proofs
of all statements and further details see \cite{inf16}.

Let $\cH_{n+1}$ denote the class of all Hamiltonians of degree $n+1$,
\begin{equation}
  H_\l(x_1,x_2) = \sum_{\abs{\alpha}\le n+1} \l_\alpha x^\alpha
\end{equation}
where $\alpha$ is a 2-multiindex. Then $\l\in\C^m$ with
$m=\frac{1}{2}(n+2)(n+3)$ provides an affine chart for
$\cH_{n+1}$. Let $\Gamma_\l$ denote the affine curve defined by the
equation $H_\l=0$.

For generic $\l$, the rank of the first homology group
$H_1(\Gamma_\l,\Z)$ is $\ell=n^2$. One may choose a set of generators
for this group over a fixed generic fibre $\l=\l_0$, and transport
them horizontally with respect to the Gauss--Manin connection to
obtain sections $\delta_1(\l),\ldots,\delta_\ell(\l)$, ramified over a
singular set $\Sigma^*\subset\cH_n$. Under a further genericity
assumption $\l\not\in\Sigma\supset\Sigma^*$, we may assume further that the first cohomology group
$H^1(\Gamma_\l,\C)$ is generated by the monomial one-forms
\begin{equation}
  \omega_\alpha = x_1 \cdot x^\alpha \d x_2 \qquad 
  0\le \alpha_1,\alpha_2 \le n-1.
\end{equation}
\begin{Def}
  The \emph{period} matrix $X(\l)$ is the $\ell\times\ell$ matrix
  \begin{equation}
    X(\l) = \begin{pmatrix}
      \int_{\delta_1(\l)} \omega_1 & \cdots & \int_{\delta_\ell(\l)} \omega_1 \\
      \vdots &                \ddots &          \vdots                 \\
      \int_{\delta_1(\l)} \omega_\ell & \cdots & \int_{\delta_\ell(\l)} \omega_\ell
    \end{pmatrix}
  \end{equation}
  defined on $\cH_{n+1}\setminus\Sigma$ and ramified over $\Sigma$.
\end{Def}

The period matrix satisfies a system of differential equations known
as the Picard--Fuchs system (or Gauss--Manin connection). The
following resut shows that this system is in fact a Q-system.

\begin{Thm}\label{thm:bny-gauss-manin}
  The period matrix satisfies the equation $\d X = \Omega X$, where
  $\Omega$ is an $(s,m,d,\ell)$-Q-system with
  \begin{equation}
    s \le 2^{\poly(n)}, \quad m \le O(n^2), \quad d \le O(n^2), \quad \ell=n^2.
  \end{equation}
\end{Thm}

\subsection{Polynomial envelopes}

Let $L$ be the linear space spanned by $r$ (possibly multivalued)
functions $f_1(t),\ldots,f_r(t)$ defined on a domain
$U\subset\C$. Denote by $\cP^k$ the space of polynomials of degree at
most $k$. By a slight abuse of notation, we also denote by $\cP^k$ a
$(k,1,1,k)$-Q-system such that the entries of its fundamental solution
matrix span the space $\cP^k$ (such a system may easily be
constructed).
\begin{Def}
  The \emph{polynomial envelope} of degree $k$ of the space $L$ is
  defined to be
  \begin{equation}
    \cP^k\otimes L = \left\{ \sum_{i=1}^r p_i(t) f_i(t)  \right\},
    \qquad p_i\in\C[t], \deg p_i\le k.
  \end{equation}
  Similarly, the \emph{polynomial envelope} of a Q-system $\Omega$ is
  defined to be $\cP^k\otimes\Omega$ (the tensor product of Q-systems
  is discussed in section~\ref{sec:transformations},
  Transformation~\ref{trans:tensor}).
\end{Def}

To establish a link between the polynomial envelope and the study of
Abelian integrals we require the following result
\cite{gavrilov:petrov-modules,sergeibook}. We use the
notation of subsection~\ref{sec:Abelian-integrals}.

\begin{Prop}
  For a generic Hamiltonian $H_\l$ and for every polynomial one-form
  $\omega$ there exist univariate polynomials $p_\alpha\in\C[t]$ and
  bivariate polynomials $u,v\in\C[x_1,x_2]$ such that
  \begin{equation}\label{eq:form-decomposition}
    \omega = \sum_\alpha (p_\alpha\circ H_\l)\cdot\omega_\alpha+u\d H_\l+\d v,
    \qquad 0\le\alpha_{1,2}\le n-1,
  \end{equation}
  where
  \begin{equation}
    \left\{ 
      \begin{array}{c}
        (n+1)\deg p_\alpha +\deg\omega_\alpha \\
        \deg v \\
        n+\deg u \\
      \end{array} \right. \le \deg\omega
  \end{equation}
\end{Prop}

Let $L_\l^e$ denote the linear space of Abelian integrals of forms of
degree at most $e$ over the Hamiltonian $H_\l$, and let $L_\l^B$
denote the linear space of Abelian integrals of the basic forms
$\omega_\alpha$.

Consider now an arbitrary polynomial one-form $\omega$ of degree at
most $e$. Let $\delta\in H_1(\{H_\l=s\},\Z)$ be a cycle on the
$s$-level surface of $H_\l$. Then $H_\l\big|_\delta\equiv s$ and
$\d H_\l \big|_\delta
\equiv0$. Integrating~\eqref{eq:form-decomposition} over $\delta$,
\begin{equation}
  \int_\delta \omega = \sum_\alpha p_\alpha(s) \int_\delta w_\alpha,
  \qquad \deg p_\alpha \le \lceil e/(n+1) \rceil.
\end{equation}
\begin{Cor}\label{cor:ai-envelope}
  For a generic Hamiltonian $H_\l$,
  \begin{equation}
    L_\l^e \subseteq \cP^{\lceil e/(n+1) \rceil}\otimes L_\l^B.
  \end{equation}
\end{Cor}

In particular, at least when the Hamiltonian is generic, $\cN(L_\l^e)$
is majorated by $\cN(\cP^{\lceil e/(n+1) \rceil}\otimes L_\l^B)$.

\section{Statement of the main result}

In this section we present the main result of the paper and deduce a
corollary concerning the zeros of Abelian integrals. We begin by
stating the general result of Petrov-Khovanskii. Our statement differs
slightly from the usual formulation in order to facilitate the analogy
to the uniform case.

To simplify the notation, when speaking about an
$(s,m,d,\ell)$-Q-system we denote by $\nu$ the number of singular
points of the system. We record the following estimate,
\begin{equation} \label{eq:singularities-estimate}
  \nu\le O(\ell^2 d).
\end{equation}
Indeed, each singular point must be a pole of one of the $\ell^2$
entries of $\Omega$, and by degree considerations each entry may admit
at most $d$ poles.
  
Let $f_1(t),\ldots,f_\ell(t)$ be $\ell$ (possibly multivalued and
singular) functions on $\C P^1$, and let $L_f$ denote the linear space
they span. Denote by $X_f$ the matrix
\begin{equation}
  X_f = \begin{pmatrix}
    f_1(t) & \cdots & f_\ell(t) \\
    f'_1(t) & \cdots & f'_\ell(t) \\
    & \vdots & \\
    f^{(\ell)}_1(t) & \cdots & f^{(\ell)}_\ell(t)
  \end{pmatrix}.
\end{equation}
Suppose that $\Omega_f=\d X_f \cdot X_f^{-1}$ is a rational matrix
function of degree $d$ which is regular and quasiunipotent.

The following result can essentially be proved by combining the
proofs of the Petrov-Khovanskii and the Varchenko-Khovanskii theorems
(see \cite{Zol}).

\begin{Thm}\label{thm:pk-envelope}
  Under the conditions of the paragraph above,
  \begin{equation}
    \cN(\cP^k\otimes\Omega_f) \le \frac{(2\nu)^{2^{\nu+1}\ell^2}-1}{2\nu-1} k +C
    \qquad \forall k\in\N,
  \end{equation}
  where $C$ is a constant depending only on $\Omega_f$ (for which a
  bound is not given). In particular, the number of zeros of a
  function in the $k$-th polynomial envelope of $L_f$ grows at most
  linearly with $k$.
\end{Thm}

The Petrov-Khovanskii result for Abelian integrals,
Theorem~\ref{thm:pk}, follows from Theorem~\ref{thm:pk-envelope} and
Corollary~\ref{cor:ai-envelope} for generic Hamiltonians. A slightly
more refined argument is needed in order to remove the genericity
assumption. We exclude this argument from our presentation, as we
shall soon see that our \emph{uniform} version of the bound
immediately extends from the generic case to the singular case.

We note that the system $\Omega_f$ arising from the formulation of
Theorem~\ref{thm:pk-envelope} satisfies the various conditions
required for a Q-system, apart from the condition of being defined
over $\Q$.  This is not a coincidence. In fact, the condition of being
defined over $\Q$ is percisely the condition responsible for the
emergence of \emph{uniform} bounds in the class of Q-systems.

We now state our main result.
\begin{Thm}\label{thm:main}
  Let $\Omega$ be an $(s,m,d,\ell)$-Q-system. Then
  \begin{equation}
    \cN(\cP^k\otimes\Omega) \le \frac{(3\nu)^{8^\nu \ell^2}-1}{3\nu-1} +s^{\exp^+(\exp^+(4^{4^\nu \ell^2}) d^5 m^5)}
  \end{equation}
\end{Thm}

Note that, in contrast to Theorem~\ref{thm:pk-envelope}, the bound in
Theorem~\ref{thm:main} is fully explicit. Also note that while
Theorem~\ref{thm:pk-envelope} applies to a particular set of
functions, Theorem~\ref{thm:main} applies to families of functions
depending (as Q-functions) on an arbitrary number of parameters $\l'$,
and the bound is uniform over the entire family.

Combining Theorem~\ref{thm:main} with Corollary~\ref{cor:ai-envelope},
we obtain an upper bound $\exp^{+2}(n^2)\cdot m+\exp^{+5}(n^2)$ for the
number of zeros of an Abelian integral of degree $e$ over a generic
Hamiltonian $H_\l$ of degree $n$. By the semicontinuity of the
counting function $\cN(\cdot)$ (see Remark~\ref{rem:semicontinuity})
this bound extends over the entire class of Hamiltonians, thus proving
Theorem~\ref{thm:main-ai}.

We note here that the implication above is a generally useful aspect
of the theory of Q-functions -- uniform bounds extend directly from
the generic case to degenerate cases. Approaches based on compactness
arguments usually require a more detailed analysis of the behavior
near the singular strata (see for instance the proof of
Theorem~\ref{thm:pk-z} in \cite{Zol}).

\section{Transformations of Q-systems}
\label{sec:transformations}

The approach employed by Petrov and Khovanskii in the proof of
Theorem~\ref{thm:pk-envelope} requires that we perform a number of
transformations to the functions being considered. Our objective is to
obtain uniform bounds by applying Theorem~\ref{thm:bny-qs}. It is
therefore necessary to prove that the appropriate transformations can
be carried it \emph{within} the class of Q-systems. In this section we
prove that this is indeed the case, and analyze the affect of each of
the transformations on the parameters $(s,m,d,\ell)$.

Let $\Omega$ denote an $(s,m,d,\ell)$-Q-system, and let $X(\cdot)$
denote a fundamental solution for $\Omega$. We assume that the base of
the system is $\C^m$, with an affine chart $\l=(t,\l')$.

\begin{Transformation}[Shift] \label{trans:shift} There exists an
  $(\hat s,\hat m,\hat d,\hat \ell)$-Q-system $\hat \Omega$ defined
  over the base space $\C^m\times\C$, with affine chart $\l\times\mu$,
  whose fundamental solution $\hat X(\cdot)$ is given by
  \begin{equation}
    \hat X(t,\l',\mu) = X(t+\mu,\l')
  \end{equation}
  and
  \begin{equation}
    \hat s= \poly(s,m,d,\ell),\quad \hat m=m+1,\quad \hat d=d,\quad \hat \ell=\ell
  \end{equation}
\end{Transformation}
\begin{proof}
  Suppose that
  \begin{equation}
    \Omega = \Omega_t(t,\l') \d t + \Omega_{\l'}(t,\l') \d \l'.
  \end{equation}
  Then
  \begin{equation}
    \hat \Omega = \Omega_t(t+\mu,\l') (\d t+\d \mu) + 
    \Omega_{\l'}(t+\mu,\l') \d \l'.
  \end{equation}
  Since $\hat\Omega$ has an explicit solution $\hat X(\cdot)$, it is
  clear that $\hat\Omega$ is integrable. It is also clear that the
  regularity and quasiunipotence of $\hat X(\cdot)$ follows from that
  of $X(\cdot)$.

  For the complexity analysis, it remains only to notice that we
  increased the dimension of the base space by one, and that the
  complexity of the formula for $\hat\Omega$ is polynomial in the
  complexity and the maximal degree of the formula for $\Omega$, the
  dimension of $\Omega$ and the dimension of the base space.
\end{proof}

We remark that it is generally not possible to perform a shifting
transformation by a specific \emph{fixed} value $\mu_0$. Indeed, the
formula for $\hat\Omega$ in this case would involve the specific value
$\mu_0$ which may be irrational, while explicit algebraic formulas by
our definitions may use only integer coefficients. We circumvent this
difficulty by extending the parameter space of the system with an
additional parameter $\mu$. Specific shifts of the system may be
obtained by restricting $\mu$ to $\mu_0$. The crucial condition which
allows this construction is that the system is not only a Q-system for
the fixed value $\mu_0$, but rather it is a Q-system with respect to
the free parameter $\mu$. This technique is generally useful in the
study of Q-systems, and has already appeared in the context of the
conformally invariant slope in \cite{inf16}.

We now consider the transformation of $\Omega$ that corresponds to
folding the $t$-plane by the transformation $w=t^2$.

\begin{Transformation}[Fold] \label{trans:fold} There exists an
  $(\hat s,\hat m,\hat d,\hat \ell)$-Q-system $\hat \Omega$ defined
  over the base space $\C^m$ with affine chart $w\times\l'$, whose
  fundamental solution $\hat X(\cdot)$ is given by
  \begin{equation}\label{eq:fold-fundamental-solution}
    \hat X(w,\l') = X(t,\l') \oplus \(t X(t,\l')\)
  \end{equation}
  where $w=t^2$, and
  \begin{equation}
    \hat s= \poly(s,m,d,\ell),\quad \hat m=m,\quad \hat d=d+2,\quad \hat \ell=2\ell
  \end{equation}
\end{Transformation}
\begin{proof}
  As in the proof of Transformation~\ref{trans:shift}, it is clear
  that $\hat\Omega$ is integrable and regular. To prove
  quasi-unipotence, let $\gamma$ be a small loop in the $(w,\l')$
  space. If $\gamma$ loops around a point with $w\neq0$ then it
  corresponds to a small loop in the $(t,\l')$ plane, and the
  monodromy of $\hat X(w,\l')=\diag(X(t,\l'),t X(t,\l'))$ around this
  loop is quasi-unipotent by the quasi-unipotence of $\Omega$. If
  $\gamma$ loops around a point with $w=0$ then $\gamma^2$ corresponds
  to a small loop in the $(t,\l')$ plane, and by the same reasoning we
  deduce that $M_{\gamma^2}$, the monodromy of $\hat X(w,\l')$ along
  $\gamma^2$, is quasi-unipotent. But $M_{\gamma^2}=M_\gamma^2$, and a
  matrix whose square is quasi-unipotent is itself
  quasi-unipotent. Thus $M_\gamma$ is quasi-unipotent as claimed.

  To explicitly define $\hat\Omega$, suppose that
  \begin{equation}
    \Omega = \Omega_t(t,\l') \d t + \Omega_{\l'}(t,\l') \d \l'.
  \end{equation}
  Then we may write
  \begin{equation}
    \hat \Omega(w,\l') = \diag( \Omega_t(t,\l') \d t + \Omega_{\l'}(t,\l') \d \l'
    , \Omega_t(t,\l') \d t + \Omega_{\l'}(t,\l') \d \l').
  \end{equation}
  Since $\d t = \d w/2t$ we may rewrite this expression in the form
  \begin{equation}
    \hat \Omega(w,\l') = \diag( \frac{\Omega_t(t,\l')}{2t} \d w 
    +\Omega_{\l'}(t,\l') \d \l',
    \frac{\Omega_t(t,\l')}{2t} \d w 
    +\Omega_{\l'}(t,\l') \d \l').
  \end{equation}
  We now replace each occurence of $t^2$ by $w$, giving an expression
  \begin{equation}
    \hat \Omega(w,\l') = \diag( \Omega_0(w,\l')+t \Omega_1(w,\l'),
    \Omega_0(w,\l')+t \Omega_1(w,\l')).
  \end{equation}
  Finally, since the second block in $\hat X$ is equal to $t$
  multiplied by the first block, we may rewrite this as
  \begin{equation}
    \hat \Omega(w,\l') = \diag( \Omega_0(w,\l')+1/t^2 \Omega_1(w,\l') ,
    \Omega_1(w,\l')+\Omega_0(w,\l')),
  \end{equation}
  which is an explicit expression for $\hat\Omega$. It is clear that
  the complexity of this expression is polynomial in $s,m,d,\ell$, the
  base space dimension is unchanged, the dimension of $\hat\Omega$ is
  $2\ell$, and the maximal degree of the coefficients of $\hat\Omega$
  is at most $d+2$.
\end{proof}

\begin{Rem}\label{rem:fold-singularities}
  If the singular points of $\Omega$ for a specific value of $\l'$
  form a set $\{s_j\}$, then the singular values of $\hat\Omega$ form
  the set $\{s_j^2\}\cup\{0,\infty\}$ since $0$ and $\infty$ are the
  two critical values of the folding map.
\end{Rem}

We next consider symmetrization of $\Omega$ around the real line. This
transformation was analyzed in \cite[3.2]{inf16}. We state here only the
result and omit the proof (which is straightforward).

For convenience we introduce the following notation. The
\emph{reflection} of a function $f(t)$ along the real line is given by
\begin{equation}
  f^\dag(t) = \overline{f(\overline t)}.
\end{equation}
If $f$ is multivalued then one may select an analytic germ of $f$ at
some point on the real line, reflect this germ, and analytically
continue the result. In cases where this choice is significant we
shall state the point of reflection explicitly. We will also use the
$\dag$ notation for vector and matrix valued functions in the obvious
way. In this paper the reflection is always taken with respect to the
time variable $t$.

\begin{Transformation}[Symmetrization] \label{trans:symm} There exists
  an $(\hat s,\hat m,\hat d,\hat \ell)$-Q-system
  $\Omega^\ominus=\hat \Omega$ defined over the same base space as
  $\Omega$, whose fundamental solution $\hat X(\cdot)$ is given by
  \begin{equation}
    \hat X(t,\l') = X(t,\l') \oplus X^\dag(t,\l'),
  \end{equation}
  and
  \begin{equation}
    \hat s= \poly(s,m,d,\ell),\quad \hat m=m,\quad \hat d=d,\quad \hat \ell=2\ell
  \end{equation}
\end{Transformation}

\begin{Rem}
  The key feature of the symmetrization transform is that the
  corresponding solution spaces $L_{\l'}(\hat\Omega)$ are closed under
  taking real and imaginary parts on the real line. Indeed, for any
  $f(t)\in L_{\l'}(\hat\Omega)$ we have also
  $f^\dag(t)\in L_{\l'}(\hat\Omega)$, and therefore
  \begin{eqnarray}
    \Re f = \frac{1}{2} \( f(t) + f^\dag(t) \) \in L_{\l'}(\hat\Omega) \\
    \Im f = \frac{1}{2i} \( f(t) - f^\dag(t) \) \in L_{\l'}(\hat\Omega)
  \end{eqnarray}
\end{Rem}

For completeness we also list the two canonical transformations of
direct sum and tensor product. Here we let $\Omega_i$ denote an
$(s_i,m,d_i,\ell_i)$-Q-system with fundamental solution $X_i(\cdot)$
for $i=1,2$, defined over a common base space. We again omit the proofs (which are straightforward).

\begin{Transformation}[Direct Sum] \label{trans:sum} There exists an
  $(\hat s,\hat m,\hat d,\hat \ell)$-Q-system $\Omega_1\oplus\Omega_2$
  defined over the same base space as $\Omega_{1,2}$, whose fundamental
  solution is given by $X_1\oplus X_2$, and
  \begin{equation}
    \hat s= s_1+s_2,\quad \hat m=m,\quad \hat d=\max(d_1,d_2),\quad \hat \ell=\ell_1+\ell_2
  \end{equation}
\end{Transformation}

\begin{Transformation}[Tensor Product] \label{trans:tensor} There
  exists an $(\hat s,\hat m,\hat d,\hat \ell)$-Q-system
  $\Omega_1\otimes\Omega_2$ defined over the same base space as
  $\Omega_{1,2}$, whose fundamental solution is given by $X_1\otimes X_2$,
  and
  \begin{equation}
    \hat s= \poly(s_{1,2},m_{1,2},d_{1,2},\ell_{1,2}),\quad
    \hat m=m,\quad \hat d=\max(d_1,d_2),\quad \hat \ell=\ell_1 \ell_2
  \end{equation}
\end{Transformation}

\begin{Rem}
  Here we use $\otimes$ to denote the tensor product of $\Omega_{1,2}$ as
  \emph{connections}, but in order to avoid confusion we note that the
  matrix form representing this connection is in fact
  $\(\Omega_1\otimes I\)\oplus \(I\otimes\Omega_2\)$.
\end{Rem}

\section{Demonstration of the main result}

In this section we present the demonstration of
Theorem~\ref{thm:main}. The proof follows the same strategy as the
Petrov-Khovanskii proof of Theorem~\ref{thm:pk-envelope}. We first
assume that all singular points of the system $\Omega$ are real. In
this case it is possible to control the variation of argument by
applying a clever inductive argument due to Petrov. For the general case, we show
that the system may be transformed to a system with real singular
points, and invoke the preceding case.

Recall that we denote by $L_\l'$ the space of all linear combinations
of solutions of the system $\Omega$ for a fixed value $\l'$, viewed as
functions of $t$ (see \eqref{eq:solution-space}).

\subsection{The case of real singular points}
In this subsection we assume that all singular points of $\Omega$ are
real.

\begin{Prop}\label{prop:real-pk}
  Let $\Omega$ be an $(s,m,d,\ell)$-Q-system, and let $\l'$ be a paramter
  such that the singular locus of the system $\Omega_{\l'}$
  is contained in $\R$. Let $r,k\in\N$ and denote
  \begin{equation}\label{eq:real-pk-f}
    f(t) = \sum_{i=1}^{r} p_i(t) f_i(t) 
    \qquad \forall i \left\{
      \begin{array}{cc}
        p_i(t) \in \R[t] \\
        \deg p_i(t) \le k \\
        f_i\in L_{\l'}
      \end{array}
    \right.
  \end{equation}
  Finally, recall that we denote by $\nu$ the number of singular points of $\Omega$. Then
  \begin{equation}
    \cN(f) \le \frac{\nu^r-1}{\nu-1} k +s^{\alpha(m,d,\ell,r)}, \qquad 
    \alpha(m,d,\ell,r) = \exp^+ (8^r \ell^{5\cdot 2^{r+1}} d^5 m^5)
  \end{equation}
\end{Prop}

\begin{figure}[h]
  \centerline{\includegraphics
    [width=0.5\textwidth]{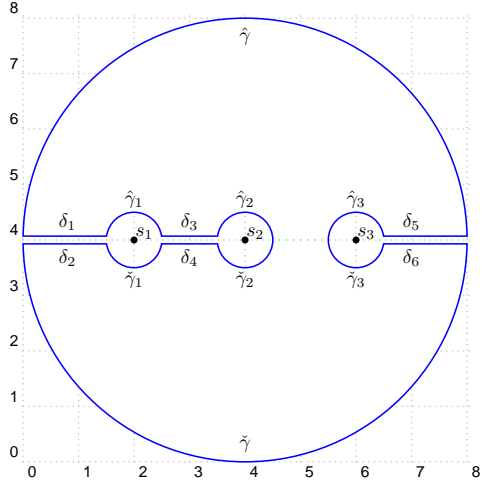}} \caption{Contour
    of integration}\label{fig:keyhole}
\end{figure}

\begin{proof}
  Let the domain $U$ and its boundary $\Gamma$, partitioned as the
  union of the curves
  $\delta_i,\hat\gamma_i,\check\gamma_i,\hat\gamma,\check\gamma$, be
  as indicated in figure~\ref{fig:keyhole} where the radius of each
  $\gamma_i$ (resp. $\gamma$) may be arbitrarily small
  (resp. large).
  Notice that one segment of the real domain is in fact
  contained in $U$ (indicated by a dotted line in the figure).
  Since any triangle avoiding the singular points can intersect
  at most one such segment, and since we can select $U$ to
  contain any single segment, it follows that to bound $\cN(f)$
  it will suffice to bound $\cN_U(f)$ independently of
  the radii defining $U$. We proceed by induction on $r$.

  When $r=1$, we have $f(t)=p_1(t) f_1(t)$. Thus by
  Theorem~\ref{thm:bny-qs}
  \begin{equation}
    \cN_U(f) \le \cN_U(f_1) + k \le C_1 + k
  \end{equation}
  where
  \begin{equation}
    C_1 = s^{\exp^+((d\ell^4 m)^5)},
  \end{equation}
  giving the desired conclusion.

  For arbitrary $r$, we proceed by applying the argument principle. We
  first rewrite $f(t)$ as
  \begin{equation}
    f(t) = \sum_{i=1}^{r} p_i(t) f_i(t) = f_1(t) F(t)
  \end{equation}
  where
  \begin{align}
    F(t) &= p_1(t) + \sum_{i=2}^{r} p_i(t) \frac{f_i(t)}{f_1(t)}
    \label{eq:F-quotient} \\
    &= p_1(t) + \abs{f_1(t)}^{-2} \sum_{i=2}^{r} p_i(t) f_i(t)
    \overline{f_1(t)}.
    \label{eq:F-product}
  \end{align}
  By Theorem~\ref{thm:bny-qs} and the argument principle,
  \begin{equation}\label{eq:f-zeros}
    \cN_U(f) = \cN_U(f_1) + \cN_U(F) \le C_1 + (2\pi)^{-1} \Var\Arg F(t)\big|_\Gamma.
  \end{equation}
  We consider the variation of argument on each piece of $\Gamma$
  separately.

  The arcs $\hat\gamma_i,\check\gamma_i$ are traversed in reverse
  orientation. Therefore we need to bound the variation of argument
  along these arcs from below. By~\eqref{eq:order-lim} the
  contribution of each arc approaches $\pi \ord F\big|_{t=s_i}$ as
  $\e\to0$. By Proposition~\ref{prop:bny-order}, the order of each
  $f_i$ is bounded in absolute value by
  \begin{equation}
    C_2 = s^{(d\ell)^{O(m)}}.
  \end{equation}
  Using~\eqref{eq:F-quotient} we deduce that
  $\ord F\big|_{t=s_i} \ge -2C_2$. Therefore
  \begin{equation}\label{eq:vararg-gamma1}
    \Var\Arg F(t)\big|_{\hat\gamma_i,\check\gamma_i} \le 2\pi C_2 
    \qquad i=1,\ldots,\nu.
  \end{equation}

  Similarly, the arcs $\hat\gamma,\check\gamma$ may be seen as small
  circular arcs around the point at infinity. We argue as above,
  noting that in this case the order of each $p_j(t)$ is bounded from
  below by $-k$. It follows that $\ord F\big|_{t=\infty} \ge
  -2C_2-k$. Therefore
  \begin{equation}\label{eq:vararg-gamma2}
    \Var\Arg F(t)\big|_{\hat\gamma,\check\gamma} \le \pi(2C_2+k).
  \end{equation}

  It remains to consider the variation of argument along the segments
  $\delta_i$. Assume that $F(t)$ is not purely real on $\delta_i$
  (otherwise there is no variation of argument). The key observation is that
  \begin{equation}
    \Var\Arg F(t)\big|_{\delta_i} \le \pi (\cN_{\delta_i} \Im_{\delta_i} F(t) + 1) 
  \end{equation} 
  where $\Im_{\delta_i}$ denotes the imaginary part taken with respect
  to the segment $\delta_i$. This fact, known as ``the Petrov trick'',
  is a simple topological consequence of the fact that the variation
  of argument of a curve contained in a half-plane is at most $\pi$.

  Using \eqref{eq:F-product} and noting that $p_j(t)$ is real on the
  real line for every $j$, we see that on $\delta_i$
  \begin{equation}\begin{split}
      \Im_{\delta_i} F(t) &= \abs{f_1(t)}^{-2} \sum_{i=2}^{r} p_i(t)
      \Im_{\delta_i} \(
      f_i(t) \overline{f_1(t)} \) \\
      &= \abs{f_1(t)}^{-2} G(t)
    \end{split}\end{equation}
  where (taking reflection with respect to $\delta_i$),
  \begin{equation}
    G(t) = \sum_{i=2}^{r} p_i(t) \Im_{\delta_i} 
    \( f_i(t) f^\dag_1(t) \).
  \end{equation}
  We used the fact that $f(\overline t)=f^\dag(t)$ on $\delta_i$.

  Let $\hat\Omega=\(\Omega\otimes\Omega^\ominus\)^\ominus$.
  Then $\hat\Omega$ is a $(\poly(s,m,d,\ell),m,d,4\ell^2)$-Q-system,
  and
  \begin{equation}
    \Im_{\delta_i} f_i(t)f^\dag_1(t) \in L_{\l'}(\hat\Omega)
    \qquad i=2,\ldots,r.
  \end{equation}
  Note that $\hat\Omega_{\l'}$ has the same singularities as
  $\Omega_{\l'}$, since the singular locus of $\Omega_{\l'}$ is
  contained in $\R$, which is the set of fixed point for the
  reflection $\dag$. We may now apply the inductive hypothesis to
  $G(t)$, since the formula defining it only involves $r-1$ summands.
  \begin{equation}\label{eq:vararg-delta}
    \begin{split}
      \Var\Arg F(t)\big|_{\delta_i} & \le \pi (\cN_{\delta_i} \Im_{\delta_i} F(t)+1)    \\
      & \le \pi (\cN_{\delta_i} G(t)+1)  \\
      & \le
      \pi\(\frac{\nu^{r-1}-1}{\nu-1} k
      +s^{\alpha(m,d,4\ell^2,r-1)} +1 \)
    \end{split}
  \end{equation}

  Using \eqref{eq:f-zeros} and summing up the variation of argument
  along $\Gamma$ using \eqref{eq:vararg-gamma1},
  \eqref{eq:vararg-gamma2} and \eqref{eq:vararg-delta} we finally
  obtain
  \begin{equation}\begin{split}
      \cN_U(f) &\le C_1 + 2\nu C_2 + (2C_2+k)
      +\nu \(\frac{\nu^{r-1}-1}{\nu-1} k +s^{\alpha(m,d,4\ell^2,r-1)} + 1  \) \\
      & \le \frac{\nu^r-1}{\nu-1} k + s^{\alpha(m,d,\ell,r)},
    \end{split}\end{equation}
  where all summands not involving $k$ are absorbed by the factor
  $s^{\alpha(m,d,\ell,r)}$ (using the estimate \eqref{eq:singularities-estimate}).

  This finishes the inductive argument.
\end{proof}

\begin{Rem}
In the proof above, we implicitly assume that  $f(t)$ does not
vanish on the boundary of $U$, so that the variation of argument
is well defined. This is a technical difficulty which can easily
be avoided. Indeed, one can define the variation of argument by
slightly deforming the boundary so that the zeros move to the
exterior of $U$, and taking the limit over the size of the
deformation. With this notion, the estimates in the proof hold
without any assumption.
\end{Rem}

\begin{Cor}\label{cor:main-real-singularities}
  Let $\Omega$ be an $(s,m,d,\ell)$-Q-system and let $\l'$ be a paramter
  such that the singular locus of the system $\Omega_{\l'}$
  is contained in $\R$. Then
  \begin{equation}
    \cN(\cP^k\otimes\Omega) \le \frac{\nu^{2\ell^2}-1}{\nu-1} k+s^{\beta(m,d,\ell)}
    \qquad \beta(m,d,\ell) = \exp^+(\exp^+(4^{\ell^2}) d^5 m^5)
  \end{equation}
\end{Cor}
\begin{proof}
  Every function $f\in L_{\l'}(\cP^k\otimes\Omega)$ may be written as
  \begin{equation}
    \begin{split}
      &f(t) = \sum_{j=1}^{r} p_j(t) f_j(t) \\
      &= \sum_{j=1}^{r} \(\Re p_j(t)\)
      f_j(t) + \sum_{j=1}^{r} \(\Im p_j(t)\) if_j(t)
    \end{split}
    \qquad \forall j \left\{
      \begin{array}{cc}
        p_j(t) \in \C[t] \\
        \deg p_j(t) \le k \\
        f_j\in L_{\l'}
      \end{array}
    \right.
  \end{equation}
  The right hand side is an expression of the
  form~\eqref{eq:real-pk-f} with $r=2\ell^2$. Applying
  Proposition~\ref{prop:real-pk} we obtain the bound stated above.
\end{proof}

\subsection{The general case}

To prove the general case, we transform the system to have real
singular points, and appeal to the result of the preceding
subsection. The transformation must be made within the class of
Q-systems, and uniform over the parameter space $\l'$.

Consider the following sequence of Q-systems $\Omega_j$. Let
$\Omega_0=\Omega$, and define $\Omega_{j+1}$ to be the system obtained
from $\Omega_j$ by applying the shifting transformation followed by
the folding transformation (we will denote the shifting parameter
introduced at this step $\mu_j$). Set $\hat\Omega=\Omega_\nu$ and
$\mu=(\mu_1,\ldots,\mu_\nu)$.

We claim that for every $\l'$, there is an appropriate choice of $\mu$
such that $\hat\Omega$ has real singularities for $(\l',\mu)$. More
specifically, we claim that for an appropriate choice of $\mu$, the
system $\Omega_j$ will admit at most $\nu-j$ non-real singularities.

To see this, we proceed by induction. The original system $\Omega$
admits at most $\nu$ singular points for any fixed value of the
parameter $\l'$. For step $j$, select some non-real singular point $s$
of $\Omega_j$ (assuming there is such a point), and set $\mu_j=-\Re
s$.
Then the shift transforms $s$ to a purely imaginary point. The
following fold transforms this point to the real line, transforms
singularities already on the real line back to the real line, and only
introduces new singularities at $0$ and $\infty$ (see
Remark~\ref{rem:fold-singularities}). This concludes the induction. A
direct computations shows that $\hat\Omega$ is a
$(\poly(s,m,\ell)^{O(\nu)},m+\nu,d+2\nu,2^\nu\ell)$-Q-system. The number of singularities
of the new system is at most $3\nu$.

We require a final preparatory lemma on the interaction between
polynomial envelopes and the folding transformation.
\begin{Lem}\label{lem:fold-envelope}
  For every value of $\l',\mu$ we have
  \begin{equation}
    L_{\l'}(\cP^{2^{k+1}-1}\otimes\Omega)\subseteq L_{\l',\mu}(\cP^k\otimes\hat\Omega)
  \end{equation}
\end{Lem}
\begin{proof}
  It clearly suffices to prove that
  \begin{equation}\label{eq:fold-envelope-step}
    L_{\l'}(\cP^{2k+1}\otimes\Omega_j)\subseteq L_{\l',\mu}(\cP^k\otimes\Omega_{j+1})
    \qquad j=1,\ldots,d-1.
  \end{equation}
  We may ignore the shift transform which (for any fixed value of
  $\mu_j$) only introduces a constant additive factor to the time
  variable and does not
  affect~\eqref{eq:fold-envelope-step}. Henceforth we assume that
  $\Omega_{j+1}$ is simply the fold of $\Omega_j$.

  Let $t$ denote the time variable of $\Omega_j$, and $w=t^2$ denote
  the time variable of $\Omega_{j+1}$. For the sake of clarity we
  write $\cP^\bullet(t),\cP^\bullet(w)$ to denote classes of
  polynomials in $t$ and $w$ respectively. Then
  \begin{equation}\begin{split}
      L_{\l'}(\cP^{2k+1}(t)\otimes\Omega_j) &=
      L_{\l'}(\cP^{k}(w)\otimes\cP^1(t)\otimes\Omega_j)
      \\
      &= L_{\l'}(\cP^{k}(w)\otimes\Omega_{j+1})
    \end{split}\end{equation}
  where the last step follows directly
  from~\eqref{eq:fold-fundamental-solution}.
\end{proof}

Finally we observe that any triangular domain $T$ in the $t$-plane
avoiding the singular locus of $\Omega$ maps under the composed
shifting and folding transforms to a domain covered by $2^{O(\nu)}$
triangles in the time domain of $\hat \Omega$. This observation,
combined with Lemma~\ref{lem:fold-envelope} and
Corollary~\ref{cor:main-real-singularities} gives
\begin{equation}\begin{split}
    \cN(\cP^k\otimes\Omega) &\le 2^{O(\nu)} \cN(\cP^k\otimes\hat\Omega) \\
    &\le \frac{(3\nu)^{8^\nu \ell^2}-1}{3\nu-1} +s^{\exp^+(\exp^+(4^{4^\nu \ell^2}) (d+2\nu)^5 (m+\nu)^5)} \\
    &= \frac{(3\nu)^{8^\nu \ell^2}-1}{3\nu-1} +s^{\exp^+(\exp^+(4^{4^\nu \ell^2}) d^5 m^5)}    
  \end{split}\end{equation}
This concludes the proof of Theorem~\ref{thm:main}.

\subsection{Concluding Remarks}
\label{sec:concluding-remarks}

The repeated-exponential nature of the bound in
Theorem~\ref{thm:main-ai} is clearly excessive. We have therefore
opted to emphasize clarity of exposition over optimality of the
analysis.  In fact, a relatively straightforward (though more
technically involved) computation using the proof of \cite{inf16} produces an
improved estimate of only four repeated exponents.

A key factor in the size of the bound is played by our construction
(following Petrov and Khovanskii) of a composite folding
transformation which moves all exisitng singularities of the system to
the real line, while only introducing new singularities at real
points. A more efficient construction of this type would yield better
estimates. We discuss a conjectural improvement of this type below.

Let $S=\{s_1,\dots,s_\nu\}\subset\C$. A polynomial $q$ is called a
\emph{folding polynomial} for $S$ if $q(S)\subset\R$ and $q$ admits
only real critical values. The change of variable $w=q(t)$, analogous
to our basic folding transformation $w=t^2$, moves the points of $S$
to the real line while only creating ramification points at the (real)
critical values of $q$. The following conjecture, in this context, has
already appeared in \cite{roitman-msc}.

\begin{Conj}
  For every $s_1,\dots,s_\nu\in\C$, there exists a folding polynomial
  $q$ of degree $O(\nu)$.
\end{Conj}

We note that the construction employed in the present paper, involving
repeated shifting and squaring, produces folding polynomials of
\emph{exponential} degree. Assuming the conjecture above, and
generalizing our treatment of Transformation~\ref{trans:fold}, it is
possible to improve our bound to a form involving only 3 repeated
exponents.

In any case, the techniques of this paper rely heavily on the results
of \cite{inf16}, and correspondingly the bounds obtained must be \emph{at
  least} doubly-exponential. It is very likely that this growth rate
is still highly excessive. Furthr improvements will probably require
completely new ideas.


\providecommand{\arxivno}[1]{http://arxiv.org/abs/#1}

\def\BbbR{$\mathbb R$}\def\BbbC{$\mathbb
  C$}\providecommand\cprime{$'$}\providecommand\mhy{--}\font\cyr=wncyr8\def\Bb%
bR{$\mathbb R$}\def\BbbC{$\mathbb
  C$}\providecommand\cprime{$'$}\providecommand\mhy{--}\font\cyr=wncyr9\def\Bb%
bR{$\mathbb R$}\def\BbbC{$\mathbb
  C$}\providecommand\cprime{$'$}\providecommand\mhy{--}\font\cyr=wncyr9\def\cp%
rime{$'$}

\end{document}